\documentclass{article}

\usepackage{chngcntr}
\usepackage{ifthen}
\usepackage{pifont}
\usepackage{bbm}
\usepackage{tikz}

\usepackage{amsmath}
\usepackage{amsthm}
\usepackage{amssymb}
\usepackage[utf8]{inputenc}
\usepackage[a4paper,portrait,margin=1in]{geometry}

\newtheorem{theorem}{Theorem}[section]
\newtheorem{corollary}[theorem]{Corollary}
\newtheorem{lemma}[theorem]{Lemma}

\newtheorem{conjecture}[theorem]{Conjecture}
\newtheorem*{claim}{Claim}
\theoremstyle{remark}

\theoremstyle{definition}
\newtheorem{definition}[theorem]{Definition}
\newtheorem{proposition}[theorem]{Proposition}

\theoremstyle{remark}

\usepackage{bbm}
\usepackage[T1]{fontenc}
\usepackage{setspace}
\usepackage{hyperref}
\onehalfspace

\begin{document}

\author{Mantas Baksys, Xuanang Chen}

\title  {On number of different sized induced subgraphs of Bipartite-Ramsey graphs}
\date{}
\maketitle

\begin{abstract}
    In this paper, we investigate the set of sizes of induced subgraphs of bipartite graphs. We introduce the definition of $C$-$Bipartite$-$Ramsey$ graphs, which is closely related to $Ramsey$ graphs and prove that in `most' cases, these graphs have multiplication tables of  $\Omega(e(G))$ in size. We apply our result to give direct evidence to the conjecture that the complete bipartite graph $K_{n,n}$ is the minimiser of the multiplication table on $n^2$ edges raised by Narayanan, Sahasrabudhe and Tomon.
\end{abstract}
\bigbreak

\setcounter{tocdepth}{1}
\tableofcontents
\newpage
\section{Introduction}

In 1960, Erdős introduced the multiplication table problem, which asks about the asymptotic order of the size of the set $M(n)= \{x\in \mathbb{N}:x=ab,0\leq a,b\leq n, a,b\in \mathbb{N}\}$. Intuitively, one might conjecture that $M(n) = \Omega (n^2)$. However, using the Hardy-Ramanujan theorem (for the exact statement and proof see \cite{l11}), Erdős in \cite{l12} showed that $M(n)=o(n^2)$. Much later, in 2008, the asymptotic order of magnitude of $M(n)$ was settled by Kevin Ford in \cite{l3} using number theoretic techniques. Specifically, we have by [Corollary 3, \cite{l3}] that $M(n) = \Theta\left({n^2}/{(\log n)^{\delta}(\log\log n)^{3/2}}\right)$, where $\delta = 1-(1 + \log \log 2)/ \log{2} \approx 0.086$.

The following generalisation of the multiplication table problem is due to Narayanan, Sahasrabudhe and Tomon in \cite{l1}. If for any simple graph $G$ we define $\mathcal{M}(G)=\{e(H):H$ is an induced subgraph of $G\}$, then we have $\mathcal {M}(K_{n,n}) = M(n)$, which creates a connection between the graph-theoretic nature of $\mathcal {M}(G)$ and number-theoretic nature of $M(n)$. Now, write $e(G)$ for the number of vertices in a graph $G$,  and $\Phi(G)=|\mathcal{M}(G)|$. It is natural to ask which bipartite graphs $G$ are extremal in terms of the behaviour of $\Phi(G)$. One way to concretely raise this question is by fixing the number of edges in the bipartite graphs under consideration, which leads to the following conjecture, which is formulated in \cite{l1}:
\begin{conjecture} \label{Conjecture 1.1}
Let $n\in \mathbb{N}$ and $G$ be a bipartite graph with $e(G) = n^2$.
Then $\Phi(G) \geq \Phi(K_{n,n})$.

\end{conjecture}
In this note, we partially prove the above conjecture for $G$ being a $C$-$Bipartite$-$Ramsey$ graph, which we define as follows:
\begin{definition} \label{definition 1.2}
For a bipartite graph $G(X,Y)$, $G$ is $C$-$Bipartite$-$Ramsey$ if it does not contain $K_{a,b}$ or $\overline{K_{a,b}}$ as an induced subgraph for any $a\geq C\log(|X|), b\geq C\log(|Y|)$.
\end{definition}

The idea to come up with this definition stems from the definition of a $C$-$Ramsey$ graph, which states that a simple graph $G$ on $n$ vertices is $C$-$Ramsey$ if it does not contain $K_m$ or $\overline {K_m}$ as an induced subgraph for $m \geq C\log n$. This set of graphs is well-studied in the literature and is well-known to satisfy certain quasi-randomness properties. Recently, Kwan and Sudakov set out to prove in \cite{l2} the following Theorem:

\begin {theorem} \label{theorem 1.3}
For any fixed $C>0$, and any $n$-vertex $C$-$Ramsey$ graph $G$, we have $\Phi(G) = \Omega (n^2)$.
\end {theorem}

Taking inspiration from Theorem \ref{theorem 1.3}, we make the following Conjecture:
\begin{conjecture} \label{conjecture 1.4}
For any fixed $C>0$, all $C$-$Bipartite$-$Ramsey$ graphs $G(X,Y)$ with $|X||Y|=m$, \\we have $\Phi (G) = \Omega (m)$.
\end{conjecture}
Using similar methods as Kwan and Sudakov, however having to do more work, when the sizes of vertex sets $X$ and $Y$ become unbalanced, we will prove the following weaker version of  Conjecture \ref{conjecture 1.4}:
\begin{theorem} \label{theorem 1.4}
 For any $C$-$Bipartite$-$Ramsey$ graph $G(X,Y)$, if there exists a constant $\alpha>0$ such that $|X||Y|=m$ and $|X|,|Y| \geq m^{\alpha}$, where $|X|$ and $|Y|$ are dependent on $m$, then $\Phi(G) = \Omega (m)$.
\end{theorem}

Note that $\Phi(G) = \Omega(m)$ is the best result we can get because of the trivial inequalities $\Phi(G) \leq e(G)$ and $e(G) \leq m$, which together imply that $\Phi(G) \leq m$.

Let us give a few remarks about Theorem \ref{theorem 1.4} and its connection to Conjecture \ref{Conjecture 1.1}. First, note that $\Phi(G) = \Omega (m)$ implies that $\Phi (G) = \Omega (e(G))$, since $e(G) \leq m$. Then, for any $m\in \mathbb{N}$, $d\leq \sqrt{m}$, $d|m$, we have $\Phi(K_{d,m/d})=\Theta\left({m}/{(\log d)^{\delta}(\log\log d)^{3/2}}\right)$. To see this, consider $H(x, y, z)$, the number of positive integers $n \in \mathbb{N}$ such that $n \leq x$ and there exists $d \in \mathbb {N}$ such that $d|n$ and $y < d \leq z$. First set $x=m/4,y=d/4,z=d/2$, then set $x=m/2^k,y=d/2^{k+1},z=d/2^k$ for non-negative integer $k$. The result follows by applying [Theorem 1(v), \cite{l3}] into the inequality 
\begin{align*}
    H(\frac{m}{4},\frac{d}{4},\frac{d}{2})\leq \Phi(K_{d,m/d}) \leq \sum_{k\geq 0} H(\frac{m}{2^k},\frac{d}{2^{k+1}},\frac{d}{2^k}),
\end{align*}
and noting that $u$ is defined in \cite{l3} as $y^{1+u}=z$, which means that $u=\log 2/\log y$ in our case. Hence, if for some function $f: \mathbb{N} \rightarrow \mathbb {N}$, $d=f(m)$, $f(m) \rightarrow \infty$ as $m \rightarrow \infty$, then  $\Phi(K_{d,m/d}) = o(m) = o(e(K_{d,m/d}))$. Specifically, letting $m=n^2$ and $d=n$, we arrive at the result that $\Phi(K_{n,n})=o(n^2)=o(e(K_{n,n}))$, which using the result of Theorem \ref{theorem 1.4} implies that for any $C$-$Bipartite$-$Ramsey$ graph $G$ with $e(G)=n^2$, $\Phi(G) \geq \Phi (K_{n,n})$. This, as claimed, gives evidence to the validity of Conjecture \ref{Conjecture 1.1}. 

But if we do not restrict our focus on the specific nature on Conjecture \ref{Conjecture 1.1}, then we are able to make even more concrete claims. If a given graph $G$ is $C$-$Bipartite$-$Ramsey$ and satisfies the condition in Theorem \ref {theorem 1.4}, then we can deduce that $\Phi(G)$ is almost surely not the minimiser of those graphs with the same number of edges as $G$, due to the fact that most positive integers $x$ have a divisor in the range $(\log x, \sqrt{x}]$. Note that this corresponds to being able to choose $f(x) \geq \log x$ for most $x$. This is because [Theorem 3, \cite{l6}] states that for $1 \leq y \leq z \leq x$, $H(x, y, z) = x\left(1+O\left(\log y/ \log z\right)\right)$. Now letting $y=\log x$, $z=\sqrt{x}$, we have that $O\left(\log y/\log z \right)=O\left( 2\log \log x/\log x\right) = o(1)$, which means that $H(x, \log x, \sqrt{x}) = x(1+o(1))$. Therefore the numbers $n \in \mathbb{N}$ such that $n\leq x$ and $n$ having a divisor between $\log x$ and $\sqrt {x}$ have asymptotic density $1$. From this, the claim that the set of numbers $x \in \mathbb{N}$ with a divisor between $(\log x, \sqrt{x}]$ have asymptotic density 1 is straight-forward, as desired.

The reason why we are interested in $C$-$Bipartite$-$Ramsey$ is because `most' random bipartite graphs are  $C$-$Bipartite$-$Ramsey$. To illustrate this idea, we give a specific example below. Consider a random bipartite graph $G(n,n)$ with probability $\frac{1}{2}$ of connecting an edge. It is almost surely $C$-$Bipartite$-$Ramsey$ for $C>5$ as $n\rightarrow \infty$: let $A$ be the event that $G$ has $ K_{5\log n,5\log n}$ or $\overline{K_{5\log n,5\log n}}$ as an induced subgraph, let $k = 5 \log n$, then
\[  \mathbb{P}(A)\leq 2\binom{n}{k}^2\left(\dfrac{1}{2}\right) ^{k^2}\leq \dfrac{2\left(n^{5\log n}\right)^2}{2^{25\left(\log n\right)^2}}\leq 2(\sqrt{e}/{2})^{25(\log n)^2}\rightarrow 0 \textit { as } n \rightarrow \infty,\] where we implicitly used the union bound.

The main goal of this paper will be to give a proof of Theorem \ref{theorem 1.4}. For this, we will split the remaining contents into 5 sections. In the first one, we will give the notation and basic definitions used in this paper. Then, we will move to give some preliminaries, which we will follow up by a discussion of density of $C$-$Bipartite$-$Ramsey$ graphs. After that, we will dedicate two sections to give and overview of the proof of Theorem \ref{theorem 1.4} and subsequently its proof. Lastly, we will give a conclusion to the paper in which we raise some new interesting questions and shed some light on existing ones.

\subsection{Notation and Basic Definitions}
We use standard notation for asymptotic estimates. Floor and ceiling symbols are generally omitted, as we are concerned about the asymptotic magnitude of the functions considered to which floor and ceiling symbols do not have any impact. We also use standard graph theoretic notations. Besides, we write the density of a graph $d(G)=e(G)/\binom{|G|}{2}$ and the density of the bipartite graph $G(X,Y)$ by $d(G)=e(G)/|X||Y|$, where $e(G)$ is the number of edges inside $G$, and $X,Y$ are two sides of it. Apart from Section 3, we use $|X|=n_1,|Y|=n_2$ or $|X|=f(m)\leq \sqrt{m}, |Y|=m/f(m)$, depending on which look nicer in different occasions. We write $N_U({x})=N(x)\cap U$ to be the neighbourhood of vertex $x$ into vertex set $U$, and let $d_U(x)=|N_U(x)|$. For a pair of vertices $\boldsymbol{v}=\{v_1,v_2\}$, write $d({\boldsymbol{v}})=d(v_1)+d(v_2)$ to be the size of $N(\boldsymbol{v})=N(v_1)\cup N(v_2)$, where the union here denotes the multiset union. Also, we work with the symmetric difference of two (multi)sets, defined for (multi)sets $A, B$ by $A\triangle B= (A\setminus B) \cup (B\setminus A)$. Lastly, let $hom(G)$ to be the largest homogeneous set (either independent or complete set) of graph $G$. The definition of $C$-$Ramsey$ graph is standard. Say a graph $G$ with $m$ vertices is $C$-$Ramsey$ if $hom(G)\leq C\log m$.

\section{Preliminaries}\label{section 2}

Recall our definition for $C$-$Bipartite$-$Ramsey$ graph given as in Definition \ref{definition 1.2}. For simplicity, throughout this section we will always assume that for a bipartite graph $G(X, Y)$ we have $|X| \leq |Y|$.

One important immediate observation, which shows the distinction between $C$-$Ramsey$ graph theory and $C$-$Bipartite$-$Ramsey$ theory is the fact that we will allow our vertex sets to vary in size relative to each other. As opposed to $C$-$Ramsey$ graphs, where each vertex can be connected to all of the other vertices, we place the condition of the graphs under consideration to be bipartite, which will severely restrict the degrees of vertices based on their vertex set. It is for this reason, that the balance of sizes of vertex sets $X$ and $Y$ will be important for the proof of Theorem \ref{theorem 1.4}. Specifically, one reason for this is that our proof strategy will involve constructing many induced subgraphs, whose sizes are sufficiently well-spaced in size, using probabilistic techniques. Note that each vertex $x \in X$ can potentially be connected to $\Omega (|Y|)$ other vertices in $Y$. Because of this, as the sizes of the vertex sets become more imbalanced, the variance in number of edges, while generating a random vertex set $U$ with expected number of edges $l$ will also increase. This will create additional difficulty in proving Conjecture \ref{conjecture 1.4}, making us instead prove the weaker Theorem \ref{theorem 1.4}, which nevertheless takes care of `most' sizes of vertex sets $X$ and $Y$. Another reason why the imbalance in the sizes of vertex sets complicates our proof is because it becomes harder to prove that these graphs $G(X,Y)$ have density in the interval $(\varepsilon, 1 -\varepsilon)$ for some absolute constant $\varepsilon > 0$.

Now, we will introduce some new definitions, adapted to bipartite graphs, which will help us to attain our result. Showing that $C$-$Bipartite$-$Ramsey$ graphs satisfy these definitions will be our way to quantify the quasi-randomness properties of the aforementioned set of graphs. Firstly, we will require a condition stating that `most' neighbourhoods of the bipartite graph are very different. 
\begin{definition}
For $c, \delta > 0$, a bipartite graph $G(X,Y)$ is $(c, \delta)$-bipartite-diverse if for each $x_1 \in X$, there exists at most $|X|^{\delta}$ vertices $x_2 \in X$ such that $|N(x_1)\triangle N(x_2)| < c|Y|$ and if for each $y_1 \in Y$, there exists at most $|Y|^{\delta}$ vertices $y_2 \in Y$ such that $|N(y_1)\triangle N(y_2)| < c|X|$.
\end{definition}
Secondly, we will also require a very similar condition to the one above, saying that if a pair of vertices has very different neighbourhoods, then `most' other disjoint pairs of vertices will also have very different neighbourhoods to the first pair.
\begin{definition}
 For  $\alpha, \delta, \varepsilon >0$, we call a bipartite graph $G(X,Y)$ $(\alpha, \delta, \varepsilon)_2$-bipartite-diverse if for each \\$\boldsymbol {x}=\{x_1, x_2\} \in \binom {X}{2}$ such that $|N(x_1)\triangle (\overline {N(x_2)}\cap Y)| \geq \alpha |Y|$ there exists at most $|X|^{\delta}$ other disjoint pairs $\boldsymbol{x'} \in \binom{X}{2}$ such that $|N(\boldsymbol{x})\triangle N(\boldsymbol{x'})| < \varepsilon |Y|$ and the corresponding statement holds with the roles of vertex sets $X$ and $Y$ reversed.
\end{definition}
Lastly, we will require another, slightly more general definition of richness for bipartite graphs, which will be the most important quasi-randomness property of $C$-$Bipartite$-$Ramsey$ graphs. One can think of it as a succinct condition, which will allow us to deduce other diversity results. We will do exactly that later on in Lemma \ref{lemma 4.6}.

\begin {definition}
For $\gamma, \delta, \varepsilon > 0$, a bipartite graph $G(X, Y)$ is $(\gamma, \delta, \varepsilon)$-bipartite-rich if for each $W_X \subseteq X$ such that $|W_X| \geq \gamma |X|$, there exists at most $|Y|^{\delta}$ vertices $y \in Y$ such that $|N(y)\cap W_X| < \varepsilon |X|$ or $|\overline{N(y)}\cap W_X| < \varepsilon |X|$ and the corresponding statement holds with the roles of vertex sets $X$ and $Y$ reversed.
\end{definition}

It must be noted that the definitions diversity and richness for $C$-$Bipartite$-$Ramsey$ graphs follow very closely from the definitions of diversity and richness for $C$-$Ramsey$ graphs by Kwan and Sudakov in \cite{l2}. The first key distinction is that we adjust the definitions to the bipartite setting by separating the conditions to each of the vertex sets. The second key distinction is that we introduce a third coefficient in our definition of richness, which is a new idea. It is largely because of this idea, that we are able to construct a proof of Theorem \ref{theorem 1.4}.

Throughout this paper, we will frequently use common probabilistic inequalities, namely Markov's inequality, Chebyshev's inequality and the Chernoff bound. The statements and proofs of all of these can be found in \cite{l13}.

Also, we will make frequent use of Turan's Theorem, and so we will give two different statements of it here:

\begin{theorem} [Turan's Theorem] \label{turan}
Let $G$ be a graph on $n$ vertices with average degree $d_G$, then $G$ has an independent set of size $\dfrac {n}{1+d_G}$.
\end{theorem}
In practice, it will be hard to find out the average degree of $G$, so we will instead use the following Corollary of Theorem \ref{turan}:
\begin{corollary} \label{turan1corollary}
Let $G$ be a graph on $n$ vertices with maximum degree $\triangle(G)$, then $G$ has an independent set of size $\dfrac {n}{1+\triangle(G)}$.
\end{corollary}
The proof of Corollary \ref{turan1corollary} follows from simply noting that $\triangle(G) \geq d_G$.

Sometimes, we will not have a good understanding of the maximum degree of the given graph $G$, but we will know the number of edges $e(G)$ within $G$, in which case we will use the second corollary of Turan's theorem:

\begin{corollary} \label{turan2corollary}
Let $G$ be a graph on $n$ vertices, then $G$ has an independent set of size at least $\dfrac{n^2}{2e(G)+n}$
\end{corollary}
Again, Corollary \ref{turan2corollary} can be easily deduced from Theorem \ref{turan} by multiplying both the numerator and denominator by $n$ and noting that $2e(G) \geq nd_G$.

\section {Density of C-Bipartite-Ramsey graphs} \label {section 3}
In \cite{l7}, Erdős proved the density for $C$-$Ramsey$ graph $G$ is such that $d(G) \in (\varepsilon, 1-\varepsilon)$ for some $\varepsilon>0$. We will need an analogous statement for $C$-$Bipartite$-$Ramsey$ graphs for our proof of Theorem \ref{theorem 1.4}. However, for $C$-$Bipartite$-$Ramsey$ graphs, the proof of the claim becomes harder, partially because of the possible imbalance in size of the two vertex sets. We were successful in proving the following result:
\begin{proposition} \label{proposition 4.4}
For any $C$-$Bipartite$-$Ramsey$ graph $G=G(X,Y)$ with $|X||Y|=m$ and $|X|,|Y|$ dependent on $m$, $G$ has density between $(\varepsilon,1-\varepsilon)$ for some $\varepsilon=\varepsilon(C,\log m/\log |X|)>0$, if:
\begin{enumerate}
    \item $|X|,|Y|\geq m^{\alpha}$ for some  $\alpha>0$ independent of $m$, or
    \item
    $|X|=o(\exp\{\sqrt{\log m}\})$
\end{enumerate}
\end{proposition}
For convenience, in the proof, let $|X|=n_1,|Y|=n_2,$ and $n_1n_2=m$. The first condition is equivalent to $\log n_2/\log n_1\leq M$ for some constant $M$ independent of $n_1,n_2$. The second condition is equivalent to $(\log n_1)^2=o(\log n_2)$. We first prove the first case.
\begin{proof}[Proof of Proposition 3.1.1.]

 When $\log{n_2}/\log{n_1}$ is bounded, it is sufficient to prove the statement in equal-sized case. To see this, note that there exists $M$ such that $n_2<n_1^M$, i.e. $\log{n_2}<M\log{n_1}$. So for any induced subgraph $G(X,Y)$ which has $n_1$ vertices on both sides, apply the equal-sized result (assume it is true) with $C'=CM$, we know that this induced subgraph has density within $(\varepsilon,1-\varepsilon)$ for some $\varepsilon=\varepsilon(C',\log m / \log |X|)>0$. As the density is within $(\varepsilon,1-\varepsilon)$ for all such induced subgraph, we know the density of $G(X,Y)$ is between $(\varepsilon,1-\varepsilon)$ as well.

So we now focus on equal-sized case. Consider $G(X,Y)$ with $|X|=|Y|=n$. If $e(G)>\dfrac{n^2}{40}$, the theorem is proved. If not, we have the following claim:
\begin{lemma} \label{lemma 4.5}
For any $k\geq 40$, let $G(n,n)$ be a bipartite graph with $n$ vertices on both sides and $e(G)\leq \dfrac {n^2}{k}$, then it contains either $K_{a,a}$ or $\overline {K_{a,a}}$ for $a\gg \dfrac{k\log(n)}{\log(k)}$.
\end{lemma}
\begin{proof}[Proof of Lemma \ref{lemma 4.5}]
Let $X$ and $Y$ be two sides of the bipartite graph. First, pick vertices in $X$ and $Y$ which have degree no more than $5n/k$, call them as $X'= \{x_1, x_2, ... , x_{m_1}\}$ and $Y'=\{y_1, y_2, ... ,y_{m_2}\}$. Let $G'=G(X'\cup Y')$ be the subgraph spanned by $X'$ and $Y'$. By a simple counting argument, we have $m_1, m_2 \geq n/2$.

Let $H$ be the largest independent graph $\overline{K_{l,l}}$ for which $l$ is the largest among all the equal-sized independent subgraph of $G^{\prime}$. We can assume that $l<c_1k\log n/\log k$ for some $c_1<1/240$, since otherwise the proof is complete. Let $H=\{x_1,...,x_l\}\cup\{y_1,...,y_l\}$ after relabelling. We have the inequality $\sum_{i=1}^{l} d(x_i) < 5nl/ k<10m_1l/k$ and an analogous statement holds for $\{y_1, y_2, ..., y_l\}$. By the same counting argument, at least $m_1/2$ vertices in $X^{\prime}$ connected to $s=20l/k$ or less vertices in $\{y_1, y_2, ... ,y_l \}$ and a respective statement holds for $Y^{\prime}$. 

We have $s<c\log n/\log k$, for $c<1/6$. Let $A_s:=\#\{A\subset \{y_1, y_2, ... ,y_l\}:|A|<s\}$ be the number of subsets of $\{y_1, y_2, ... ,y_l\}$ with fewer than $s$ elements. We then have:
\[ A_s \leq s\binom{l} {s-1} < s^{2} {\left(\dfrac {l}{s}\right)}^{s}e^{s} < s^{2}{\left (\dfrac{ek}{20} \right )}^{20l/k}< s^{2}k^{c\log n/\log k} < n^{1/3}, \]
 where in the second inequality we used Stirling's approximation and later made use of the fact that $c<1/6$, $s\ll \log n$ and $s-1\leq l/2$.
 
By pigeonhole principle, there must exist $A=\{a_1, a_2, ... , a_{p_1}\} \subset \{x_1, x_2, ... , x_l\}$ with $p_1<s$ and $B = \{b_1, b_2, ... , b_{q_1}\} \subset Y^{\prime} \setminus \{y_1, y_2, ... , y_l\}$ with $q_1 > ({m_2/2 - l})/{n^{1/3}} > n^{1/2} $, such that each vertex in $B$ is connected to all vertices in $A$ but no other vertices in $\{x_1, x_2, ... , x_l\} \setminus A$. Similarly, obtain the set $C \subset X^{\prime} \setminus \{x_1, x_2, ... , x_l\}$ of size $q_2 > n^{ {1}/{2}}$ and set $D \subset \{y_1, y_2, ... , y_l\}$ of size $p_2 < s$ by reversing the roles of $X^{\prime}$ and $Y^{\prime}$ in the previous argument. 

Let $G^{\prime \prime}=G(B\cup C)$ be the graph spanned by $B$ and $C$. Then $G^{\prime \prime}$ does not have $\overline{K_{s,s}}$ as an independent subset, since if it does, then consider $\overline{K_{s,s}}\cup (\overline H\setminus (A\cup D))$ which is an independent subgraph of $G$ with both sides of size more than $l$, contradicting to the fact that $l$ is the largest size among all the equal-sized independent subgraph of $G^{\prime}$.
\begin{center}
    \begin{tikzpicture}
    \foreach \x in {1,...,20} {
                \draw[very thin,opacity=0.5] (0,-0.7+\x/15+rand*0.01) -- (4,2.4+\x/30+rand*0.01);
                \draw[very thin,opacity=0.5] (0,-0.7+\x/15+rand*0.01) -- (4,2.45+\x/33+rand*0.01);
                \draw[very thin,opacity=0.5] (0,-0.7+\x/15+rand*0.01) -- (4,2.50+\x/38+rand*0.01);
                \draw[very thin,opacity=0.5] (0,2.4+\x/30+rand*0.01) -- (4,-0.7+\x/15+rand*0.01);
                \draw[very thin,opacity=0.5] (0,2.4+\x/33+rand*0.01) -- (4,-0.7+\x/15+rand*0.01);
                \draw[very thin,opacity=0.5] (0,2.4+\x/38+rand*0.01) -- (4,-0.7+\x/15+rand*0.01);

            }
            
    \foreach \x in {1,...,10} {
                \draw[very thin,opacity=0.5] (0,-0.7+\x/7.5+rand*0.1) -- (4,-0.7+\x/7.5+rand*0.1);

            }
    \draw[fill=white] plot [smooth cycle,tension=1]                       coordinates {(0,-0.68)(0.3,0)(0,0.65)(-0.3,0)};
    \draw[fill=white] plot [smooth cycle,tension=1]                       coordinates {(4,-0.68)(4.3,0)(4,0.65)(3.7,0)};
    \node at (0,0) {$C$};
    \node at (4,0) {$B$};
    \draw[fill=white] plot [smooth cycle,tension=1]                       coordinates {(0,2.4)(0.15,2.73)(0,3.07)(-0.15,2.73)};
    \draw[fill=white] plot [smooth cycle,tension=1]                       coordinates {(4,2.4)(4.15,2.73)(4,3.07)(3.85,2.73)};
    \node at (0,2.73) {$A$};
    \node at (4,2.73) {$D$};
    \draw plot [smooth cycle,tension=1]                       coordinates {(0,1.7)(0.3,2.5)(0,3.2)(-0.3,2.5)};
    \draw plot [smooth cycle,tension=1]                       coordinates {(4,1.7)(4.3,2.5)(4,3.2)(3.7,2.5)};
    \node at (-0.8,2.3) {$H\cap X$};
    \node at (4.8,2.3) {$H\cap Y$};

    \end{tikzpicture}

\end{center}
\begin{center}
    No induced $\overline{K_{s,s}}$ in $G''=G(B\cup C)$
\end{center}

Now, we claim that if a bipartite graph $G(\sqrt{n},\sqrt{n})$ with $\sqrt{n}$ vertices on both sides doesn't contain $\overline{K_{s,s}}$, then it must contain $K_{m,m}$ for some $m\gg \dfrac{ck\log(n)}{\log(k)}$.

From \cite{l8}, Hattingh and Henning define the Bipartite Ramsey Number to be the smallest integer $b=b(p,q)$ such that any bipartite graph with $b$ vertices on both sides will contain either $K_{p,p}$ or $\overline{K_{q,q}}$. They proved that $b(p,q)\leq \binom{p+q}{p}$. By setting $p=ck\log(n)/\log(k)$, $q=c\log(n)/\log(k)$ and applying Stirling's Formula, we have:
\[ \binom{p+q}{p}=\Theta\left(\dfrac{\Big(\dfrac{ck\log(n)}{\log(k)}\Big)^{\tfrac{c\log(n)}{\log(k)}}e^{\tfrac{c\log(n)}{\log(k)}}}{\sqrt{2\pi \dfrac{c\log(n)}{\log(k)}}\Big(\dfrac{c\log(n)}{\log(k)}\Big)^{\tfrac{c\log(n)}{\log(k)}}} \right)=O\left((ke)^{\tfrac{c\log(n)}{\log(k)}}\right)=O(n^{c+\tfrac{c}{\log(k)}})=o(\sqrt{n}) \tag{1} \label{equation 1}\]

Thus, if $G(n,n)$ contains $\overline{K_{a,a}}$ as induced subgraph for $a\gg \dfrac{k\log n}{\log k}$, we're done. If not, we can find its induced subgraph $G'(\sqrt{n},\sqrt{n})$ which doesn't contain $\overline{K_{s,s}}$ as induced subgraph for $s\leq c\log n/\log k$ and any $c<1/6$. Then by the argument above, for $m\geq ck\log n/\log k$ we have $b(s,m)\leq \binom{s+m}{s}=o(\sqrt{n})$, so $G(n,n)$ must contain $K_{m,m}$ as induced subgraph. We're done in both cases.
\end{proof}

From Lemma \ref{lemma 4.5} and the argument before, the first claim of Proposition \ref{proposition 4.4} is proved.
\end{proof}
\begin{proof}[Proof of Proposition 3.1.2.]
Now, we finish the proof of the second claim of Proposition \ref{proposition 4.4}. When $\log n_2/\log n_1$ is unbounded and $\log n_2/\log n_1> (\log{n_1})^{1+\varepsilon}$ for any $\varepsilon>0$, applying essentially the same argument in Lemma \ref{lemma 4.5}, replacing $n$ by $n_1,n_2$ appropriately, and taking $c=1/20<1/6$ in the proof, what we are left to prove is:
\begin{claim}
If $G(X,Y)$ is a bipartite graph with $|X|=\sqrt{n_1},|Y|=\sqrt{n_2}$, while $G(X,Y)$ does not contain  $\overline{K_{\log n_1/20\log k,\log n_2/20\log k}}$ as induced subgraph, then it contains $K_{C\log n_1,C\log n_2}$ as induced subgraph.
\end{claim}
To prove the claim, we use a different method. Choose $k$ to be the smallest positive integer such that $k/\log k>C$. From ($\ref{equation 1}$), taking $n=n_1$, $c=1/20$, then $c+c/\log k<n^{1/10}$, so for any induced subgraph $H\subseteq G(X,Y)$ with $n_1^{1/10}$ vertices on both sides, $H$ either contains induced $\overline{K_{s,s}}$ or $K_{m,m}$, for $s=\log n_1/20\log k$, $m=k\log n_1/20\log k$. 

There are at least $\sqrt{n_2}/n_1^{1/10}$ many such $H$ in $G(X,Y)$. Firstly, suppose $\sqrt{n_2}/{2n_1^{1/10}}$ of them contain $K_{m,m}$ as an induced subgraph. By pigeon-hole principle, at least $\log n_2/\log n_1$ such $H$ will have the same vertex set in $X$ among $\log n_2/\log n_1\binom{n_1^{1/10}}{\log n_1/20\log k}+1\leq n_1^{k\log(n_1)/10\log k}\log n_2/\log n_1 $ many of them. Their union contains $K_{k\log n_1/20\log k,k\log n_2/20\log k}$ as induced subgraph, therefore it has $K_{C\log n_1,C\log n_2}$ as an induced subgraph, since $k/\log k>C$. To get $n_1^{k\log(n_1)/10\log k}\log n_2/\log n_1 $ many $H$, the number of vertices in $Y$ should be
\begin{align*}
     n_1^{k \log(n_1)/10\log(k) + 1/10}\log(n_2)/\log(n_1) &=\exp\left((\log n_1)^{2}\dfrac{k}{10\log k}+\dfrac{\log(n_2)}{\log(n_1)}+\dfrac{1}{10}\log(n_1)\right) \\
      & \leq \exp\left(3\max\left\{(\log n_1)^{2}\dfrac{k}{10\log k},\dfrac{\log(n_2)}{\log(n_1)},\dfrac{1}{10}\log(n_1)\right\} \right) \\ & \leq \exp\Big(\dfrac{1}{4}\log n_2\Big) \\ &\leq \dfrac{\sqrt{n_2}}{2}
\end{align*}

The inequality holds as $(\log n_1)^2=o(\log n_2)$, and $k=k(C)$ is a constant. This is a contradiction to the assumption that $G(X,Y)$ does not have $K_{C\log n_1,C\log n_2}$ as induced subgraph. For the other case, there are at least $\sqrt{n_2}/{2n_1^{1/10}}$ many $H$ has $K_{m,m}$ as induced subgraph. Apply the same argument again, we would get the result that $G(X,Y)$ has $\overline{K_{\log n_1/20\log k,\log n_2/20\log k}}$ as an induced subgraph. We are done in both cases.
\end{proof}
We are unable to show the similar density result when $\exp\{ \sqrt{\log m}\}\ll f(m)\ll m^{\alpha}$, yet we suspect the result will also hold in this case.

\section{Overview of the proof of Theorem \ref{theorem 1.4}}
Our proof of Theorem \ref{theorem 1.4} is based on the following key construction. We claim, we can always construct sufficiently many sets of induced subgraphs, which are well-separated in size, using probabilistic methods. Specifically, it will be sufficient to show that for a given function $g(m)$, in each set of $O(g(m))$ consecutive positive integers in the range $[0, e(G)]$, we can construct $\Omega (g(m))$ induced subgraphs with different sizes in this range. For the immediate implication of the Theorem \ref{theorem 1.4}, we will need to prove that $e(G) = \Omega (m)$ for $C$-$Bipartite$-$Ramsey$ graphs with vertex sets of polynomial size in $m$ and the result will follow by picking $\Omega(m/g(m))$ different sets of $O(g(m))$ consecutive integers in the aformentioned range, which do not intersect pairwise. In our proof, we will prove this exact claim for a carefully chosen function $g(m)$.

In \cite{l2}, Kwan and Sudakov proved a very similar claim, although concerning $C$-$Ramsey$ graphs $G$ on $n$ vertices. They showed, that one can construct $\Omega(n^{3/2})$ subgraphs on different sizes in intervals of length $O(n^{3/2})$, which do not intersect pairwise. A well, known fact in $C$-$Ramsey$ theory is that $e(G)=\Omega(n^2)$, which gives the result that for all $C$-$Ramsey$ graphs, $\Phi(G)= \Omega(n^2)$. Kwan and Sudakov followed a construction, based on randomly generating a vertex set $U$, such that $e(U)$ is not too far away from it's expected value. Then, using richness and diversity of $C$-$Ramsey$ graphs, they construct another vertex set $W$, whose vertices have similar degree, yet there exist many vertices whose degrees are well-separated. It is exactly these properties of $W$, which will enable Kwan and Sudakov to construct sufficiently many induced subgraphs with different sizes by considering the values of $e(U\cup Z)$, for different choices of $Z\subseteq W$.

In our proof, we largely follow the same construction as Kwan and Sudakov in \cite{l2}, yet do this for a different set of graphs. First, it is important to make the trivial observation that bipartite graphs on $n$ vertices have an independent set of size at least $n/2$, which immediately implies that they are not $C$-$Ramsey$. However, one would expect that given a similar condition as $C$-$Ramsey$, it should be possible to show that bipartite graphs have `many' induced subgraphs of different sizes. This is exactly the intuition, which leads us to come up with the definition $C$-$Bipartite$-$Ramsey$ to prove Theorem \ref{theorem 1.4}.

When it comes to new ideas in our proof, we had to non-trivially alter definitions of diversity and richness, since we're dealing with a significantly different set of graphs. Most importantly, as remarked above, we will allow vertex sets of bipartite graphs considered be imbalanced, which will  complicate the proof as it will cause higher variance in our probabilistic method of generating a vertex set $U$ with a set number of edges. This will mean that, the other vertex set $W$ will need to be chosen significantly larger in size than using the methods of \cite{l2} allows, in order for us to construct sufficiently many subgraphs of different sizes. For this reason, we will need to introduced a new definition of given bipartite graph being $(\gamma,  \delta, \varepsilon)$-bipartite-rich, which will allow us to generate $W$ to be as large as required by choosing $\delta$ to be sufficiently small. Also, note that \cite {l2} only considers the set of vertices with unique degrees into $U$ in order to construct $W$, which would contain vertices, which are well-separated in degree. We have realised that this condition can be significantly relaxed in order to acquire $W$ significantly larger in size and we'll make use of this observation in our proof.

In practice, we will generate $W$ as the union of disjoint vertex sets $S$, $T$ and $Z$, all of which will serve different purposes. We will require a condition, that all vertices in $W$ have similar degree into $U$ for better control, yet all degrees of vertices in $S$ into $U$ will be sufficiently larger then degrees of all vertices in $T$ into $U$. By picking different subsets of $S \cup T$ and all vertices in $U$, we will already be able to generate many induced subgraphs of different sizes. Finally, we will require $Z$ to have many vertices of different degrees, which will allow us to finish the proof of Theorem \ref{theorem 1.4} by adding different vertices of $Z$ to previously generated induced subgraphs. It is important to note that we will only allow vertices of $Y$ to be in $W$, since $Y$ will generally be much bigger in size. Since our graph $G$ is bipartite, this will give us an extremely valuable condition that there are no edges between vertices in $W$, which will lead to significant simplification in the proof of Theorem \ref{theorem 1.4}, when compared to the approach of \cite{l2}.

Lastly, it is important to point out that to keep the size of $W$ as large as possible, sometimes $W$ will need to be taken as a set of disjoint pairs of vertices. This will require a separate definition of diversity, but will not cause any other major changes in the proof. Also, the reader should think of richness and diversity, as special properties of $C$-$Bipartite$-$Ramsey$ graphs. It is precisely these properties that enable us to use probabilistic arguments in order to efficiently find many sets of induced subgraphs, well-separated in size. In the case, when for all $\alpha > 0$, $|X| < m^{\alpha}$ our methods break down because of the presence of the logarithm in the definition of $C$-$Bipartite$-$Ramsey$. We choose to continue the use of the logarithm from $C$-$Ramsey$ theory, as choosing functions of higher asymptotic magnitude would cause major problems, while deducing the quasi-randomness properties, necessary for our proof. Although, we expect that other methods could be used in this case, when $|X|$ is sub-polynomial to give similar results. This is because we expect $\Phi(G)$ to increase in size, as vertex sets become more imbalanced, keeping the edge count of $G$ fixed. One can see this heuristically by noting that as the size of $|X|$ decreases, each vertex in $Y$ is on average connected to a smaller number of vertices in $X$, allowing for smaller differences in size between induced subgraphs.

\section {Proof of Theorem \ref{theorem 1.4}}

As promised before, we now deduce other diversity conditions from richness.
\begin{lemma} \label{lemma 4.6}
Assume a bipartite graph $G(X,Y)$ is $(\gamma, \delta, \varepsilon)$-rich and $\gamma <  1/2$. Then, the following statements all hold:
\begin{enumerate}
    \item G is $(\delta, \varepsilon/2)$-bipartite-diverse.
    \item G is $(\alpha\varepsilon/2, \delta, \alpha)_2$-bipartite-diverse for $\alpha \geq 2\gamma$.
    \item There are at most $|Y|^{1+\delta}$ disjoint pairs of vertices $\boldsymbol {y}= \{y_1, y_2\} \in \binom {Y}{2}$ such that the inequality $|N(y_1)\triangle (\overline {N(y_2)}\cap X)|<\varepsilon|X|/2$ holds, and there are at most $|X|^{1+\delta}$ disjoint pairs of vertices $\boldsymbol {x}=\{x_1, x_2\} \in \binom {X}{2}$ such that $|N(x_1)\triangle (\overline {N(x_2)}\cap Y)| <  {\varepsilon |Y|}/{2}$ holds.
\end{enumerate}
\end{lemma}
\begin{proof}
Without loss of generality, we prove each statement only for vertices of $X$ only, as we can conclude the full proof by exchanging the roles of $X$ and $Y$.
For the first statement, note that for each $x_1 \in X$, either $|N(x_1)|\geq |Y|/2$ or $|\overline {N(x_1)}\cap Y| \geq |Y|/2$. Since $\gamma < 1/2$,  in the former case, there are at most $|X|^{\delta}$ vertices $x_2 \in X$ such that $|\overline{N(x_2)}\cap N(x_1)| < \varepsilon |N(x_1)| \leq \varepsilon |Y|/2$ and in the latter case, there are at most $|X|^{\delta}$ vertices $x_2 \in X$ such that $|N(x_2)\cap (\overline {N(x_1)}\cap Y)|< \varepsilon |\overline{N(x_1)} \cap Y|\leq \varepsilon |Y|/2$.
In each case, there are at most $|X|^{\delta}$ vertices $x_2 \in X$ such that $|N(x_1) \triangle N(x_2)| < \varepsilon |Y|/2$, since $\overline {N(x_2)}\cap N(x_1) \subseteq N(x_1)\triangle N(x_2)$ and $N(x_2)\cap (\overline{N(x_1)}\cap Y) \subseteq N(x_1)\triangle N(x_2)$.

For the second statement, let $\boldsymbol {x} = \{x_1, x_2\}$ and note that if $|N(x_1)\triangle (\overline {N(x_2)} \cap Y)| \geq \alpha |Y|$, then either $|N(x_1)\cap N(x_2)| \geq \alpha |Y|/2$ or $|\overline {N(x_1)} \cap \overline {N(x_2)}\cap Y| \geq \alpha |Y|/2$. Without loss of generality, assume that $|N(x_1)\cap N(x_2)| \geq \alpha |Y|/2 \geq \gamma |Y|$, since $\alpha \geq 2\gamma$ .
For contradiction, suppose that there exists a set $Z$ of at least $|X|^{\delta}$ disjoint pairs $\boldsymbol {x'} \in \binom {X} {2}$ such that $|N(\boldsymbol{x})\triangle N(\boldsymbol{x'})|< \alpha \varepsilon |Y|/2$. Then, for each vertex $x'$ in a given $\boldsymbol{x'}\in Z$, we have $|\overline {N(x')}\cap N(x_1) \cap N(x_2)| \leq |N(\boldsymbol{x}) \triangle N(\boldsymbol {x'})| \leq \alpha \varepsilon |Y|/2 \leq \varepsilon |N(x_1)\cap N(x_2)|$, which contradicts $(\gamma, \delta, \varepsilon)$-bipartite-richness, as $|N(x_1)\cap N(x_2)| \geq \gamma |Y|$. Hence, the claim follows.

For the third statement, we will prove that for each of $|X|$ choices in $x_1 \in X$ there are at most $|X|^{\delta}$ vertices $x_2 \in X$ for which $|N(x_1)\triangle (\overline{N(x_2)}\cap Y)|< \varepsilon |Y|/2$. Again, note that either $|N(x_1)| \geq |Y|/2$ or $|\overline {N(x_1)}\cap Y| \geq |Y|/2$. In the former case, there are at most $|X|^{\delta}$  vertices $x_2 \in X$ for which $|N(x_1)\cap \overline {N(x_2)}|< \varepsilon |N(x_1)| \leq \varepsilon |Y|/2$ and in the latter case, there are again at most $|X|^{\delta}$ vertices $x_2 \in X$ such that $|\overline {N(x_1)}\cap \overline{N(x_2)} \cap Y| < \varepsilon |\overline {N(x_1)} \cap Y| \leq \varepsilon |Y|/2$. Note that $N(x_1) \cap \overline {N(x_2)} \subseteq N(x_1) \triangle (\overline {N(x_2)} \cap Y)$ and $\overline {N(x_1)}\cap \overline {N(x_2)} \cap Y \subseteq N(x_1) \triangle (\overline {N(x_2)}\cap Y)$ to have that in each case, the required claim follows.
\end{proof}

Now, we will prove that $C$-$Bipartite$-$Ramsey$ graphs satisfy our bipartite-richness definition.

\begin{lemma} \label{lemma 4.7}
For any $C$-$Bipartite$-$Ramsey$ graph $G(X, Y)$, with $|X|=f(m)$ and $|Y|=m/f(m)$ such that for some $\alpha > 0$, $m^{\alpha} \leq f(m) \leq \sqrt m$, then for any $\delta >0$, there exists $\gamma = \gamma (C,\delta, \alpha) > 0$ with $\varepsilon = 4\gamma$ such that $G$ is $(\gamma,  \delta, \varepsilon)-bipartite-rich$.
\end{lemma}
\begin {proof}
Suppose, for contradiction that $G$ fails to be $(\gamma, \delta, 4\gamma)$-bipartite-rich for all $\gamma > 0$. As before, it will suffice to prove richness in one direction, and the other direction will be obtained by exchanging the roles of $X$ and $Y$ and making the obvious adjustments. Hence,  we can assume that a set $W_X \subseteq X$ with $|W_X| \geq \gamma |X|\geq\gamma f(m)$ and there exist a set $Z_Y \subseteq Y$ with $|Z_Y| \geq |Y|^{\delta}$ contradicting $(\gamma, \delta, 4\gamma)$-bipartite-richness, which means that $\forall v \in Z_Y, |N(v)\cap W_X| < 4\gamma|W_X|$ or $|\overline {N(v)} \cap W_X| < 4 \gamma |W_X|$. Without loss of generality, assume that for $T_Y \subseteq Z_Y$ of size $|T_Y| \geq |Y|^{\delta}/2$ such that $\forall v \in T_Y, |N(v)\cap W_X| < 4\gamma|W_X|$. Now let $h(m) = \min (\gamma m^{\alpha}, (m/f(m))^{\delta}/2)$ and note that $h(m) \geq m^{\beta}$ for some $\beta=\beta(\alpha,\delta) >0$. By a simple counting argument, there exist subsets $S_X \subset W_X$ and $S_Y \subset Z_Y$ such that $|S_X|=|S_Y|=\sqrt {h(m)}$ for $m$ large enough such that $d(S_X,S_Y)<8\gamma$. Then, by Lemma $\ref{lemma 4.5}$ $G[S_X, S_Y]$  contains either $K_{a, a}$ or $\overline {K_{a, a}}$ as a subgraph with $ a\geq \frac {\beta \log m}{3840\gamma \log (1/8\gamma)}$ and choosing $\gamma$ dependent on $C$ and $\beta
= \beta (\delta, \alpha)$ one can find a $\gamma = \gamma (C, \delta, \alpha)$ small enough, so that $a\geq C\log m$, which contradicts the $C$-$Bipartite$-$Ramsey$ property as $G[S_X, S_Y]$ is a subgraph of $G$ and $G$ does not contain $K_{C\log(f(m)),C\log (m/f(m))}$ or $\overline {K_{C\log f(m),C\log (m/f(m))}}$ as subgraph, which gives the desired contradiction.
\end{proof}

It will be convenient to deduce Theorem \ref{theorem 1.4} from the following Lemma:
\begin{lemma} \label{lemma 4.8}
Let $G(X,Y)$ be an $C$-$Bipartite$-$Ramsey$ graph with $|X|=f(m)$ and $|Y|=m/f(m)$, such that there exists $\alpha >0$ for which $m^{\alpha} \leq f(m) \leq \sqrt m$. Then there exist $c=c(C,\alpha)>0$ such that for any $l$ with $cm \leq l \leq 2cm$, there exist disjoint subsets $U\subseteq X\cup Y$ and $W \subseteq Y$, for which $|e(U)-l|=O(\frac {m}{\sqrt {f(m)}})$ and $|W|= O(m/f(m)^{\frac {3}{2}})$. Also, $|\{e(U\cup W') : W' \subseteq W\}| = \Omega (\frac {m}{\sqrt{f(m)}})$.
\end {lemma}
Given Lemma \ref{lemma 4.8}, the proof of our theorem is easy, since each degree of a vertex in $W$ is at most $f(m)$, so $|e(U\cup W) - e(U)| = O(\frac {m}{\sqrt{f(m)}})$. This means that $\Omega (\frac {m}{\sqrt{f(m)}})$ values in the multiplication table are within $O(\frac {m}{\sqrt {f(m)}})$ centred at $e(U)$ , which means that since $|e(U)-l|=O(\frac {m}{\sqrt{f(m)}})$, we can pick $\Omega (\sqrt{f(m)})$ values of $l$, since $e(G) = \Theta (m)$ by Proposition \ref{proposition 4.4},  to get $\Omega (m)$ different values in the multiplication table. \qed

To prove Lemma \ref{lemma 4.8}, we will need to construct an elaborate construction by first randomly generating the set $U$ such that for a given $l$ such that $cm \leq l \leq 2cm$, one has $|e(U)-l|=O(\frac {m}{\sqrt {f(m)}})$ and secondly, for each such $U$, we will need to construct the set $W$of size at most $O(m/f(m)^{\frac {3}{2}})$ such that $|\{e(U\cup W') : W' \subseteq W\}| = \Omega (\frac {m}{\sqrt{f(m)}})$.

To start the construction, we will need the following lemma:
\begin{lemma} \label{lemma 4.9}
For a $C$-$Bipartite$-$Ramsey$ graph $G$ and any $l$ with $cm \le l \le 2cm$ , there exist sets $U, S, T, Z$ satisfying the following claims:
\begin{enumerate}
\item $|e(U)-4l|=O(\frac {m} {\sqrt{f(m)}})$ and $U$ is an ordinary set of vertices. 
\item $S, T, Z$ are disjoints sets of vertices,  or disjoint sets of pairs of vertices, also disjoint from $U$. Also, $S\cup T \cup Z = O(m/f(m)^{\frac {3}{2}})$ 
\item $\forall \boldsymbol {v} \in S\cup T \cup Z, d_U(\boldsymbol {v}) = d + O(\sqrt{f(m)})$, where $d=\Theta (f(m))$. 
\item $\min_{\boldsymbol {x} \in S} d_U(\boldsymbol {x}) - \max_{\boldsymbol {x} \in T} d_U(\boldsymbol {x}) = \Omega (\sqrt {f(m)})$.
\item The degrees from $Z$ into $U$ are distinct $($that is $\forall \{\boldsymbol {z_1}, \boldsymbol{z_2}\} \in Z$, $d_U(\boldsymbol{z_1}) \neq d_U(\boldsymbol {z_2})$ if $\boldsymbol{z_1} \neq \boldsymbol{z_2})$.
\item $|Z| = \Omega (\sqrt{f(m)}),|S| =  \Omega (\sqrt {f(m)}),|T| = \Omega (m/f(m)^{\frac {3}{2}})$ or $|S| = \Omega(\sqrt {f(m)}), |T| = \Omega (m/f(m)^{\frac {3}{2}})$.
\end{enumerate} 
\end{lemma}

\begin{proof}
We prove them in a certain order. Firstly, from Lemma \ref{lemma 4.7}, there exists $\gamma = \gamma (C, \alpha, \delta)$ such that for $n$ large enough, $G$ is $(\gamma, \delta, 4\gamma)$-bipartite-rich, with $\delta = \alpha/5$.
Hence, from Lemma \ref{lemma 4.6}, our graph $G$ is $(\delta, 2\gamma)$-bipartite-diverse and $(4\gamma^{2}, \delta, 2\gamma)_2$-bipartite-diverse. Now consider the $\Omega (m^2/ f(m)^2)$ sums $d(y_1) + d(y_2)$ for $\{y_1, y_2\} \in \binom {Y}{2}$, all of which lie in the range $[0, 2f(m)]$. By Pigeonhole principle, there exists some $d'=\Theta (f(m))$ such that there exists $W\subset \binom {Y}{2}$ of size $|W| = \Omega (m^2 / f(m)^{\frac {5}{2}})$ such that for all pairs $\{w_1, w_2\} \in W$, we have $d(w_1) + d(w_2) = d' + O({\sqrt {f(m)}})$. By Lemma \ref{lemma 4.7}, we can delete at most $|Y|^{1+\delta}= {(\frac{m}{f(m)})^{1+\delta}} =o(m^2/f(m)^{\frac {5}{2}})$ pairs of $\{y_1, y_2\} \in \binom {Y}{2}$ from $W$ to obtain $W' \subseteq W$, where $W' = \Omega (m^2/f(m)^{\frac {5}{2}})$ and for $\{w_1, w_2\} \in W$ we have $|N(w_1)\cap \overline {N(w_2)}| < 2\gamma|X|$. Note that $|Y|^{1+\delta}= {(\frac{m}{f(m)})^{1+\delta}} =o(m^2/f(m)^{\frac {5}{2}})$ since $f(m) \leq \sqrt {m}$ and $\delta \leq \frac {1}{10}$. Now interpret $W'$ as a graph on the vertex set $Y$ and hence have that either there exists a vertex with degree $\Omega (m/f(m)^{\frac {5}{4}})$ or there exists a matching of size $\Omega (m/f(m)^{\frac {5}{4}})$. In the former case, let this vertex be denoted by $y\in Y$. Then, set $d''=d' - d_G(y)$ and let $L$ be the set of neighbours of $y$ in $W'$. In the latter case, let $L$ be the set of pairs comprising the matching and $d'' = d'$. Then in both cases $|L| = \Omega (m/f(m)^{\frac {5}{4}})$ and $\forall \boldsymbol {y} \in L, d(\boldsymbol {y}) = d'' + O(\sqrt {f(m)})$. Next, let $F \subseteq \binom {L}{2}$ be the set of $\{\boldsymbol {y_1}, \boldsymbol {y_2}\} \in \binom {L}{2}$ such that $|N(\boldsymbol {y_1}) \triangle N(\boldsymbol {y_2})| < 4\gamma^{2}|X|$ (note that we can assume $\gamma < \frac {1}{4}$). Then, in each case using a different diversity assumption, if we interpret $L$ as a graph with edges given by $F$, the graph has maximal degree $|Y|^{\delta}$ and so using Turan's Theorem (Corollary \ref{turan1corollary}), our graph has an independent set $A$ of size $\Omega (\frac {|L|}{|Y|^{\delta}}) = \Omega (\frac {m}{f(m)^{5/4}}\frac {f(m)^{\delta}}{m^{\delta}}) = \Omega (m/f(m)^{\frac {3}{2}})$, for all $\delta \leq \frac {\alpha}{4(1-\alpha)}\leq \frac {\alpha}{5}$. Hence in the original graph, for every $\{ \boldsymbol {y_1}, \boldsymbol {y_2}\} \in \binom {A}{2}$, $|N(\boldsymbol {y_1}) \triangle N(\boldsymbol {y_2})|=|\Omega(f(m))|$ . 

Let $c=c(C,\alpha)>0$ be such that $e(G) \geq 800cm$ and let $cm \leq l \leq 2cm$, while let $p=\sqrt {\dfrac{4l}{e(G)}} \in (0, 0.1)$ be the probability of picking a vertex in $X\cup Y$ independently to be in $U$.
    
    \begin{claim}
    The following each hold with probability greater than $0.8$.
    \begin{enumerate}
        \item $|e(U)-4l|=O(\frac {m}{\sqrt {f(m)}})$;
        \item There is Q $\subseteq A$ involving no vertices of $U$ with $|Q|\geq \dfrac{2}{3}|A|$;
        \item There is $R\subseteq A$ with $|R|\geq \dfrac{2}{3}|A|$ and $d_{U}(\boldsymbol{x})=pd''+O(\sqrt {f(m)})$ for each $\boldsymbol{x}\in R$;
        \item $|N_{U}(\boldsymbol{x})\triangle N_{U}(\boldsymbol{y})|=\Omega(f(m))$ for each $\{\boldsymbol{x},\boldsymbol{y}\}\in\binom{A}{2}$;
        \item The equality $d_{U}(\boldsymbol{x})=d_{U}(\boldsymbol{y})$ holds for $O(m^2/f(m)^{\frac {7}{2}})$ pairs $\{\boldsymbol{x},\boldsymbol{y}\}\in\binom{A}{2}$
    \end{enumerate}
    \end{claim}
    \begin{proof}
    \begin{enumerate}
    \item Note that $\mathbb{E}(e(U))=4l$ and Var$(e(U))=O(m^2/f(m))$, as each edge shares endpoints with at most $O(m/f(m))$ other edges. By Chebyshev's inequality, $|e(U)-4l|=O(\frac {m}{\sqrt {f(m)}})$ with probability at least $0.99$ for large enough implied constant in the $O(\frac{m}{\sqrt {f(m)}})$ term.

    \item This follows directly from our definition of picking $U$: $\mathbb{E}(|\{r\in A, r\notin U\}|)\geq(1-p)^{2}|A|$, variance is $O(|A|)$. Recall that $(1-p)^2\geq 0.81$, the result follows from Chebyshev's inequality.
    
    \item For each $\boldsymbol{x}\in A$ we have $\mathbb{E}d_U(\boldsymbol{x})=pd''+O(\sqrt{f(m)})$ and the variance is $O(f(m))$, by Chebyshev with probability at least 0.999 we have $d_U(\boldsymbol{x})=pd''+O(\sqrt {f(m)})$. Let $R$ be the set of $\boldsymbol{x}$ satisfying the condition as in the claim, using Markov's inequality we acquire the result. Now set $d=pd''$ and have that $d=\Theta (f(m))$.
    \item Recall that $|N(\boldsymbol{x})\triangle N(\boldsymbol{y})|=\Omega(f(m))$. Since $|N_{U}(\boldsymbol{x})\triangle N_{U}(\boldsymbol{y})|=|(N(\boldsymbol{x})\triangle N(\boldsymbol{y}))\cap U |$, we know that $|N_{U}(\boldsymbol{x})\triangle N_{U}(\boldsymbol{y})|$ has a binomial distribution with parameters $|N(\boldsymbol{x})\triangle N(\boldsymbol{y})|$ and $p$. By Chernoff bound, $\mathbb{P}(|N_{U}(\boldsymbol{x})\triangle N_{U}(\boldsymbol{y}|<(p/2)|N(\boldsymbol{x})\triangle N(\boldsymbol{y}|)=e^{-\Omega(f(m))}\rightarrow 0$ as $f(m)\rightarrow \infty$, and the desired result follows directly.
    
    \item Since $|N(\boldsymbol{x})\triangle N(\boldsymbol{y})|=\Omega(f(m))$, we let $|N(\boldsymbol{x})\setminus N(\boldsymbol{y})|=a,$ $|N(\boldsymbol{y})\setminus N(\boldsymbol{x})|=b$. Without loss of generality assume that $a\leq b$ and note that $a+b\geq c'f(m)$, for some $c' >0$, such that $b\geq \dfrac{c'f(m)}{2}$. Then, 
    \begin{align*}
        \mathbb{P}(d_{U}(\boldsymbol{x})=d_{U}(\boldsymbol{y}))&=\mathbb{P}(|N_{U}(\boldsymbol{x})/ N_{U}(\boldsymbol{y})|=|N_{U}(\boldsymbol{y})/ N_{U}(\boldsymbol{x})|)\\&=\sum_{k=1}^a p^k(1-p)^{a-k}\binom{a}{k}p^k(1-p)^{b-k}\binom{b}{k}\\&\leq \max_{i\in \{1,2,...,b\}} p^i(1-p)^{b-i}\binom{b}{i}\\&=O\left(p^{pb+\lambda}(1-p)^{b-pb - \lambda }\binom{b}{pb+\lambda}\right)
    \end{align*}
    for some $\lambda \in (-1, 1)$, since the mode of a binomial distribution with parameters $b$ and $p$ is $\lfloor {pb + p} \rfloor$ or $\lceil{pb +p} \rceil - 1$, hence it can be expressed as $pb + \lambda$. Then, $O(p^{pb+\lambda}(1-p)^{b-pb - \lambda }\binom{b}{pb+\lambda})=O(\frac{1}{\sqrt{p(1-p)b}})=O(1/\sqrt {f(m)})$ by applying Stirling's formula. Now, the expected number of pairs $\boldsymbol{x}, \boldsymbol{y} \in \binom {A}{2}$ such that $d_U(\boldsymbol{x})=d_U(\boldsymbol{y})$ is $O(m^2/f(m)^{\frac {7}{2}})$ and the desired result follows from Markov's inequality.
    \end{enumerate}
    \end{proof}
 Now fix  an outcome of $U$ satisfying all 5 of the above properties, arbitrarily divide $R\cap Q$ which has size at least $\dfrac{|A|}{3}$ into two subsets $H$ and $P$ of size $\Omega(m^2/f(m)^{\frac {3}{2}})$. Consider the new graphs on the vertex sets $H$ and $P$, where $\boldsymbol{x}\boldsymbol{y}$ is an edge if and only if $d_{U}(\boldsymbol{x})=d_{U}(\boldsymbol{y})$. In both cases, this new graph has $O(m/f(m)^{\frac {7}{2}})$ edges, so by Turan's theorem (Corollary \ref{turan2corollary}) $H$ has an independent set $B$ and $P$ has an independent set $Z$ both of size $\Omega(\sqrt {f(m)})$. Now order the degrees $d_{U}(\boldsymbol {b})$ for $\boldsymbol{b} \in B$ in an increasing order. If the upper $|B|/2$ elements in ordering of $B$ have $\Omega(m/f(m)^{\frac {3}{2}})$ neighbours in $H$, then we let $T$ be the vertex set of this upper half elements and the vertices in $H$ of the same degrees as them, let $S$ to be the lower $|B|/3$ elements in $B$. If not, then the lower $|B|/2$ elements in $B$ has $\Omega(m/f(m)^{\frac {3}{2}})$ many neighbours in $H$, then we let $S$ be the vertex set of this lower half elements and the vertices in $H$ of the same degrees as them, let $T$ to be the upper $|B|/3$ elements in $B$. In either case, we will get $T,S$ of size $\{\Omega (m/f(m)^{\frac {3}{2}}),\Omega(\sqrt {f(m)})\}$ (in some order), and $\min_{\boldsymbol {x}\in T}d(\boldsymbol{x})-\max_{\boldsymbol {x}\in S}d(\boldsymbol{x})\geq |B|/6=\Omega(\sqrt{f(m)})$, as required.

 \end{proof}
 
 Given the construction of Lemma \ref{lemma 4.9}, we make use of it to define the following. Without loss of generality, assume that $|S|=\Omega(m/f(m)^{\frac {3}{2}})$ and $|T|=\Omega (\sqrt {f(m)})$ and fix the orderings of their elements. Now let $c'' > 0$ be such that
 \begin{align*}
 \min_{\boldsymbol {y} \in T} d_{U}(\boldsymbol{y}) - \max_{\boldsymbol{y}\in S} d_{U}(\boldsymbol{y}) \geq 8c''\sqrt {f(m)},\\ |T|\geq c''\sqrt{f(m)} \textit { and } |S| \geq  2c''m/f(m)^{\frac {3}{2}}
 \end{align*}
 Then, let $\mathcal{P}$ be the set of pairs $(k,i)\in \mathbb{Z}^2$ such that $c'm/f(m)^{\frac {3}{2}} \leq k \leq 2c'm/f(m)^{\frac {3}{2}}$ and $0 \leq i \leq c'\sqrt {f(m)}$. Then, for each $(k,i)\in \mathcal {P}$, define $Q_{k,i}$ as the union of the first $(k-i)$ elements of $S$ and the first $i$ elements of $T$. For ease of notation, we will let $U_{k,i}=U\cup Q_{k,i}$.
 
 We already have that for all $k$ and $i$ in the appropriate specified range, where all of the implied constants are positive:
 \begin{align}
    e(U \cup Q_{k, 0}) - e(U \cup Q_{k-1, 0}) = \Theta (f(m))\\
    e(U \cup Q_{k,i}) - e(U \cup Q_{k, i-1}) = \Theta (\sqrt {f(m)})
\end{align}
The statement in $(1)$ trivially comes from the fact that for all $\boldsymbol {y} \in S$, $d_{U}(\boldsymbol {y}) = \Theta (f(m))$. For $(2)$ we use that $\min_{\boldsymbol {y} \in T} d_{U}(\boldsymbol{y}) - \max_{\boldsymbol{y}\in S} d_{U}(\boldsymbol{y}) = \Theta (\sqrt {f(m)})$ and observe that $e(Q_{k,i}) = 0$ for all pairs $(k, i) \in \mathcal{P}$, since $G$ is bipartite and $S, T$ only contain vertices of $Y$.
Now letting $k$ and $i$ range through $\mathcal {P}$, we already have $\Omega (|S|\sqrt{f(m)}) = \Omega (\frac {m}{f(m)})$ unique values in $\Phi(G)$. To prove Theorem \ref{theorem 1.4}, we will need to improve this number to $\Omega(\frac {m}{\sqrt{f(m)}})$, as remarked before.

Therefore, the final building block of the proof of Lemma \ref{lemma 4.8} will be to show that we can induce $\Omega(\sqrt{f(m)})$ more subgraphs in magnitude, which are of different sizes for each $U$ and $W$. This will be done by separating out the values of $e(U_{k,i})$ by at least $Q\sqrt{f(m)}$, for $Q$ being a sufficiently large constant and using that $Z$ is a set of $\Omega(\sqrt{f(m)})$ vertices of different degrees into $U$, which will allow us to achieve the exact improvement.

\begin {proof}[Proof of Lemma \ref{lemma 4.8}]
    First, let $\mathcal{P'} \subseteq \mathcal {P}$, where $(k, i) \in \mathcal{P}$ is in $\mathcal{P'}$ if $e(U_{k+1, 0}) - e(U_{k, i}) = \Omega (\sqrt{f(m)})$ for some positive implied constant. Now by $(1)$ and $(2)$, we have that for each $k$ in the following range $[c'm/f(m)^{\frac {3}{2}}, 2c'm/f(m)^{\frac {3}{2}}-1]$, there are $\Omega(\sqrt{f(m)})$ such $i$ that $(k, i) \in \mathcal {P'}$ and so $|\mathcal{P'}| = \Omega (m/f(m))$. Now, let $\mathcal {I} = \{e(U_{k,i}) : (k,i) \in \mathcal{P'}\}$ to have that $|\mathcal{I}| = \Omega (m/f(m))$ and each pair of elements of $\mathcal{I}$ are separated by at least $\Omega (\sqrt{f(m)})$. Now let $Q>0$ be a constant such that for every $\boldsymbol{x} \in W$, $d_U(\boldsymbol{x}) \in [d - Q\sqrt{f(m)}, d + Q\sqrt{f(m)}]$. Then, order $\mathcal{I}$ in an increasing way and form $\mathcal{I'} \subseteq \mathcal{I}$ by picking every $q$th element in $\mathcal{I}$ for some constant $q=q(Q)>0$, so that every pair of elements in $\mathcal{I'}$ are separated by $2Q\sqrt{f(m)}$. Note that we still have that $|\mathcal{I'}| = \Omega (m/f(m))$ because the gaps between consecutive elements of $\mathcal{I}$ are $\Omega(\sqrt{f(m)})$. Now for each $\boldsymbol{z} \in Z$, $d_{U_{k,i}}(\boldsymbol{z}) = d_U(\boldsymbol{z})$, as $\{\boldsymbol {z}\}, Q_{k,i} \subset Y$ and $G$ is bipartite. Hence, for each $(k, i) \in \mathcal {P'}$ and each $\boldsymbol {z} \in Z$, $e(U_{k,i} \cup \boldsymbol{z})$ are distinct, which means that we've constructed $\Omega(m/\sqrt{f(m)})$ induced subgraphs of different sizes, as required.
\end {proof}

\section{Conclusion}

We have given a proof of a weaker version of Conjecture \ref{conjecture 1.4} in the form of Theorem \ref{theorem 1.4}, and we suspect that Conjecture \ref{Conjecture 1.1} is indeed true, and if so the proof would involve a combination of graph theoretic and number theoretic results in a major way. A good starting point would be to prove that the density is within $(\varepsilon,1-\varepsilon)$ for some $\varepsilon>0$ for all sub-polynomial $C$-$Bipartite$-$Ramsey$ graphs.

Although our treatment of $C$-$Bipartite$-$Ramsey$ graphs gives evidence to Conjecture \ref{Conjecture 1.1}, there is still a long way to go in order to prove it. To illustrate this, we give a somewhat extremal example implied by Conjecture \ref{Conjecture 1.1}.
\begin {conjecture} \label {conjecture} For any $n \in \mathbb{N}$, $\mathcal{M}(K_{n^2, (n+1)^2}) \geq \mathcal {M}(K_{n(n+1), n(n+1)})$
\end {conjecture}
Clearly, to prove this conjecture, clearly it must be shown that:
\[|\mathcal {M}(K_{n^2, (n+1)^2}) \setminus \mathcal {M}(K_{n(n+1), n(n+1)})|\geq |\mathcal {M}(K_{n(n+1), n(n+1)}) \setminus \mathcal {M}(K_{n^2, (n+1)^2})| \] 
However, to prove the above statement, one would likely require to show existence of a prime number in the range $(n(n+1), (n+1)^2)$, which is stronger than the Legendre's Conjecture. Because of this difficulty even in specific cases, we expect Conjecture \ref{Conjecture 1.1} to be very difficult at present.

$\mathbf{Acknowledgements}$. We would like to thank Dr.\ Aled Walker for his supervision, encouragement and helpful comments.

%
%

\textsc{University of Cambridge, Centre for Mathematical Sciences, Wilberforce Rd, \\Cambridge
CB3 0WA}

\textit{Email address}: \texttt{baksysmantas@gmail.com}

\textsc{University of Cambridge, Centre for Mathematical Sciences, Wilberforce Rd, \\Cambridge
CB3 0WA}

\textit{Email address}:
\texttt{shawnchen177@gmail.com}

\end{document}